[1994/12/01]
\documentclass{ijmart}

\usepackage{amsmath, graphicx,amssymb,amsfonts,amscd}

\chardef\bslash=`\\ 

\newtheorem{thm}{theorem}[section]
\newtheorem{corollary}[thm]{Corollary}
\newtheorem{lemma}[thm]{lemma}
\newtheorem{proposition}[thm]{proposition}

\theoremstyle{definition}

\newtheorem{remark}{Remark}[section]

\theoremstyle{remark}

\newcommand{\eval}[2][\right]{\relax
  \ifx#1\right\relax \left.\fi#2#1\rvert}

\begin{document}
\title[Quasi-energy function]{A quasi-energy function for Pixton diffeomorphisms defined by generalized Mazur knots\thanks{The research was done with the support of Russian National Foundation (project
23-71-30008).}}
\author[T. Medvedev]{Timur Medvedev}
\address{Laboratory of Algorithms and Technologies for Network Analysis\\
HSE University\\
136 Rodionova Street\\
Niznhy Novgorod, Russia 
}
\email{mtv2001@mail.ru} 
\author[O. Pochinka]{Olga Pochinka}
\address{International Laboratory of Dynamical Systems and Applications\\ HSE University\\
25/12 Bolshaya Pecherckaya Street\\
Niznhy Novgorod, Russia}

\date{Received on MONTH, YEAR}
\issueinfo{VOL}{NUM}{MONTH}{YEAR}
\doiinfo{10.1007/DOI-NUMBER}

\begin{abstract}
In this paper we give a lower estimate for the number of critical points of the Lyapunov function for Pixton diffeomorphisms (i.e. Morse-Smale diffeomorphisms in dimension 3 whose chain recurrent set consists of four points: one source, one saddle and two sinks).
Ch.~Bonatti and V.~Grines proved that the class of topological equivalence of such diffeomorphism  $f$ is completely defined by the equivalency class of the Hopf knot $L_{f}$ that is the knot in the generating class of the fundamental group of the manifold $\mathbb S^2\times\mathbb S^1$. They also proved that there are infinitely many such classes and that any Hopf knot can be realized by a Pixton diffeomorphism. D.~Pixton proved that diffeomorphisms defined by the standard Hopf knot $L_0=\{s\}\times \mathbb S^1$ have an energy function (Lyapunov function) whose set of critical points coincides with the chain recurrent set whereas the set of critical points of any Lyapunov   function for Pixton diffeomorphism with nontrivial (i.e. non equivalent to the standard) Hopf knot is strictly larger than the chain recurrent set of the diffeomorphism. The Lyapunov function for Pixton diffeomorphism with minimal number of critical points is called the quasi-energy function. In this paper we construct a quasi-energy function for Pixton diffeomorphisms defined by a generalized Mazur knot.
\end{abstract}

\maketitle
\tableofcontents

\section{Introduction and the main results}
\label{intro}
	Let $M^n$ be a smooth closed $n$-manifold with a metric $d$ and let $f:M^n\to M^n$ be a diffeomorphism.  For two given points $x,y\in M^n$ a sequence of points  
$x=x_0,\dots,x_m=y$ is called an {\em $\varepsilon$-chain of length $m\in\mathbb N$} connecting $x$ to $y$ if $d(f(x_{i-1}),x_{i})<\varepsilon$ for $1\leqslant i\leqslant m$ (Fig. \ref{chain}).
\begin{figure}[h]\centerline{\includegraphics[height=5cm]{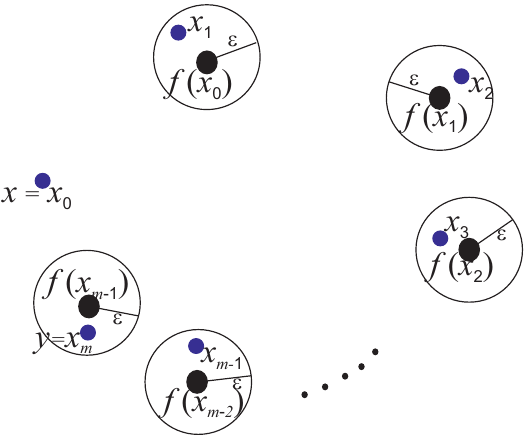}}\caption{\small An $\varepsilon$-chain of length $m\in\mathbb N$}\label{chain}\end{figure}
	
	A point $x\in M^n$ is called {\em chain recurrent} for the diffeomorphism $f$ if for every $\varepsilon>0$ there is an $\varepsilon$-chain of length $m$ connecting $x$ to itself for some $m$ ($m$ depends on $\varepsilon>0$). The {\em chain recurrent set}, denoted by $\mathcal R_{f}$, is the set of all chain recurrent points of $f$. Define the equivalence on $\mathcal R_{f}$ by the rule:  $x\sim y$ if for every  $\varepsilon>0$ there is are $\varepsilon$-chains connecting $x$ to $y$ and $y$ to $x$. This equivalence relation defines equivalence classes called {\em chain
recurrence classes} or {\em chain components}.
	
	If the chain recurrent set of a diffeomorphism $f$ is finite then it consists of periodic points. A periodic point $p\in\mathcal R_f$ of period  $m_p$ is said to be {\em hyperbolic} if absolute values of all the eigenvalues of the Jacobian matrix $\left(\frac{\partial f^{m_p}}{\partial x}\right)\vert_{p}$ are not equal to 1. If absolute values of all these eigenvalues are greater (less) than 1 then  $p$ is called a {\em sink} (a {\em source}). Sinks and sources are called {\em nodes}. If a hyperbolic periodic point is not a node  then it is called a {\em saddle}.
	
	Let $p$ be a hyperbolic periodic point of a diffeomorphism $f$ whose chain recurrent set is finite. The {\em Morse index} of $p$, denoted by $\lambda_p$, is the number of eigenvalues of Jacobian matrix whose absolute values are greater than 1. The {\em stable manifold} $W^s_p=\{x\in M^n:\lim\limits_{k\to+\infty}d(f^{km_p}(x),p)=0\}$ and the {\em unstable manifold} $W^u_p=\{x\in M^n:\lim\limits_{k\to+\infty}d(f^{-km_p}(x),p)=0\}$ of $p$ are smooth manifolds diffeomorphic to $\mathbb R^{\lambda_p}$ and $\mathbb R^{n-\lambda_p}$, respectively.  Stable and unstable manifolds are called {\em invariant manifolds}. 
	A connected component of the set $W^u_p\setminus p$ (resp. $W^s_p\setminus p$) is called a {\em unstable} (resp. {\em stable}) {\em separatrice} of  $p$.
	
	A diffeomorphism $f:M^n\to M^n$ is called a {\em Morse-Smale} diffeomorphism if
	\begin{enumerate}
	
\item its chain recurrent set $\mathcal R_f$ consists of finite number of hyperbolic  points;
	
	\item for any two points $p$, $q\in\mathcal R_f$ the manifolds $W^s_p$, $W^u_q$ intersect transversally. 	
\end{enumerate}
	
	C.~Conley in \cite{Conley1978} gave the following definition: a {\it Lyapunov function} for a Morse-Smale diffeomorphism $f:M^n\to M^n$ is a continuous function $\varphi:M^n\to\mathbb{R}$ satisfying
	\begin{itemize}
		\item $\varphi(f(x))<\varphi(x)$ if $x\notin R_f$;
		\item $\varphi(f(x))=\varphi(x)$ if $x\in R_f$.
	\end{itemize}
He proved the existence of a   Lyapunov function for arbitrary Morse-Smale diffeomorphism.

Notice that every Morse-Smale diffeomorphism $f$ has a {\em Morse-Lyapunov function}\footnote{This function can be constructed, for example, by {\em suspension}. Consider the topological flow $\hat f^t$ on the manifold $M^n\times\mathbb R$ defined by $\hat f^t(x)=x+t$. Define the diffeomorphism $g:M^n\times\mathbb R\to M^n\times\mathbb R$ by $g(x,\tau)=(f(x),\tau-1)$ and let $G=\{g^k~,k\in\mathbb Z\}$ and $W=(M^n\times\mathbb R)/G$. Denote by $p_{_{W}}:M^n\times\mathbb R\to W$ the natural projection and denote by $f^t$ the flow on $W$ defined by $f^t(x)=p_{_{W}}(\hat f^t(p^{-1}_{_{W}}(x)))$. The flow $f^t$ is called the {\em suspension over $f$}. By construction the chain recurrent set of $f^t$ consists of the finite number of periodic orbits $\delta_i=p_{_{W}}(\mathcal O_i\times\mathbb R),\,i\in\{1,\dots, k_f\}$ and this means that the suspension $f^t$ is a Morse-Smale flow. A Lyapunov function for these flows is constructed in \cite{Meyer1968}. Then the restriction of this function on $M$ is the desired Lyapunov function for $f$.},
i.e. a Lyapunov function $\varphi:M^n\to\mathbb R$ which is a continuous Morse function\footnote{Recall a notion a continuous Morse function on  $M^n$ given by M.~Morse \cite{morse1959topologically}. 	Let $\varphi:M^n\to\mathbb R$ be a continuous function with real values. A point $p\in M^n$ is said to be {\it regular} if there is a neighborhood $V_p\subset M^n$ of $p$ and a homeomorphism onto its image $\phi_p:y\in V_p\mapsto \phi_p(y)=(x_1(y),\cdots, x_n(y))\in\mathbb R^n$ such that 
	$$x_i(p) = 0, i\in\{1,\cdots, n\},\,\,\varphi(y)=\varphi(p)+x_n(y),\,y\in V_p.$$ Otherwise the point $p$ is called {\it critical}. Denote by $Cr_\varphi$ the set of critical points of $\varphi$. A critical point $p$ is said to be {\it non-degenerate critical point of index $\lambda_p$} if there are coordinates $x_i,i\in\{1,\cdots, n\}$ of the critical point $p$ and $\lambda_p\in\{0,\cdots,n\}$ such that  $$\varphi(y)=\varphi(p)-\sum\limits_{i=1}^{\lambda_p}x^2_i(y)+\sum\limits_{i=\lambda_p+1}^nx^2_i(y),\,\,y\in V_p.$$ A Morse function whose every critical point is non-degenerate is called a {\it continuous Morse function}.} such that each periodic point $p\in\mathcal R_f$ is its non-degenerate critical point of index $\lambda_p$ with Morse coordinates $(V_p,\phi_p:y\in V_p\mapsto (x_1(y),\dots,x_n(y))\in\mathbb R^n$ and $$\phi_p^{-1}(Ox_1\dots x_{\lambda_p})\subset W^u_p,\,\,\phi_p^{-1}(Ox_{\lambda_p+1}\dots x_n)\subset W^s_p.\quad\quad\quad(*)$$
Obviously, the hyperbolic periodic points of $f$ need to be critical for its Morse-Lyapunov function. 	If a Morse-Lyapunov function $\varphi$ of $f$ has no critical points outside $\mathcal R_f$ then following \cite{Pixton1977} we call it an {\em energy function} for the Morse-Smale diffeomorphism $f$.
	
	The proof of existence of an energy Morse function for a Morse-Smale diffeomorphism of the circle is an easy exercise. D.~Pixton \cite{Pixton1977} in 1977 proved that every Morse-Smale diffeomorphism of a surface has an energy function. There he also constructed an example of a Morse-Smale diffeomorphism on the 3-sphere which admits no energy function. The obstacle to existence of an energy function in his example was the {\em wild embedding} of the saddle separatrices in the ambient manifold (i.e. the closure of the separatrice is not a submanifold of the ambient space). From \cite{MePo2018} it follows that there are Morse-Smale diffeomorphisms with no energy function on manifolds of any dimension $n>2$. Therefore, following \cite{GrLauPo2009} for a Morse-Smale diffeomorphism $f$ we call a Morse-Lyapunov function with the minimal number of critical points (denote it by $\rho_{_f}$) a {\em quasi-energy function}. 
	Notice that $\rho_{_f}$ is a topological invariant, i.e. if two diffeomorphisms $f,f':M^n\to M^n$ are topologically conjugate (that is there is a diffeomorphism $h:M^n\to M^n$ such that $h\circ f=f'\circ h$) then $\rho_{_f}=\rho_{_{f'}}$. Indeed, if $\varphi$ is a quasi-energy function for $f$ then $\varphi'=\varphi h^{-1}$ is a quasi-energy function for $f'$. 
	
In this paper we give a lower estimate of $\rho_{_f}$ for Pixton diffeomorphisms. The class of {\it Pixton diffeomorphisms} $\mathcal P$ is defined in the following way.  Every diffeomorphism $f\in\mathcal P$ is a Morse-Smale 3-diffeomorphism whose chain recurrent set consists of four points: one source, one saddle and two sinks (for details see section \ref{Pico}). Notice that Pixton's example is a diffeomorphism of this class.
	According to \cite{BoGr} the class of topological conjugacy of a diffeomorphism $f\in\mathcal P$ is completely defined by the equivalence class of the {\em Hopf knot} $L_{f}$ associated to $f$,   i.e. a knot in the generating class of the fundamental group of the manifold $\mathbb S^2\times\mathbb S^1$ (see Proposition \ref{Pidy}). Moreover, any Hopf knot can be realized as a Pixton diffeomorphism.
	
Recall that a {\it knot} in $\mathbb S^2\times\mathbb  S^1$ is a smooth embedding  $\gamma:\mathbb S^1\to \mathbb S^2\times\mathbb S^1$ or the image of this embedding $L=\gamma(\mathbb S^1)$. Two knots $L,L'$ are said to be {\it equivalent} if there is a homeomorphism $h:\mathbb S^2\times\mathbb S^1\to \mathbb S^2\times\mathbb S^1$ such that $h(L)=L'$ and whose induced isomorphism $h_*:\pi_1(\mathbb S^2\times\mathbb S^1)\to\pi_1( \mathbb S^2\times\mathbb S^1)$ acts as the identity map.  Two knots $\gamma,\gamma'$ are {\it smoothly homotopic} if there exists a smooth map  $\Gamma:\mathbb S^1\times[0,1]\to \mathbb S^2\times\mathbb S^1$ such that  $\Gamma(s,0)=\gamma(s)$ and $\Gamma(s,1)=\gamma'(s)$ for every $s\in\mathbb S^1$. If $\Gamma|_{\mathbb S^1\times\{t\}}$ is an embedding for every  $t\in[0,1]$ then the knots are said to be {\it isotopic}. 
	
Any Hopf knot $L \subset \mathbb S^2 \times\mathbb S^1$ is smoothly homotopic to the standard Hopf knot $L_0=\{s\}\times \mathbb S^1$ (see, for example, \cite{K-L}) but generally it is neither isotopic nor equivalent to it. B.~Mazur \cite{Mazur} constructed the Hopf knot  $L_M$ which we call the {\em Mazur knot} and which is non-equivalent and non-isotopic to $L_0$ (see Fig. \ref{not}).
\begin{figure}[h]\centerline{\includegraphics
[width=11 cm]{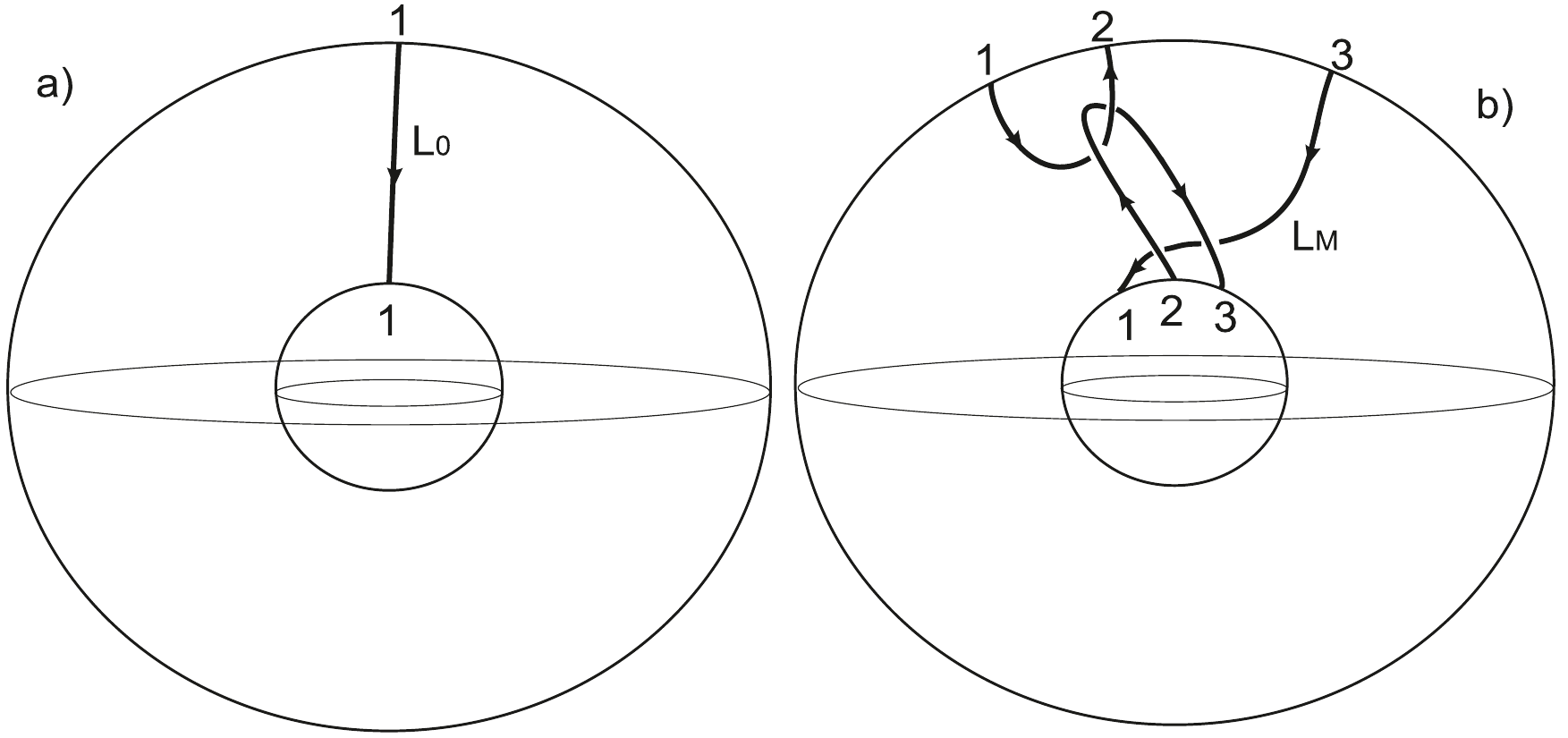}}\caption{\small Two non-isotopic and non equivalent Hopf knots $L_0$ and $L_M$: a) the standard Hopf knot $L_0$; b) the Mazur knot $L_{M}$}\label{not}\end{figure}
It follows from the results of \cite{AMP} that there exists a countable family of pairwise non-equivalent Hopf knots $L_{n},\,n\in\mathbb N$ which are {\em generalized Mazur knots} (Fig.~\ref{mn}).
\begin{figure}[h]\center{\includegraphics[width=
0.5\linewidth]{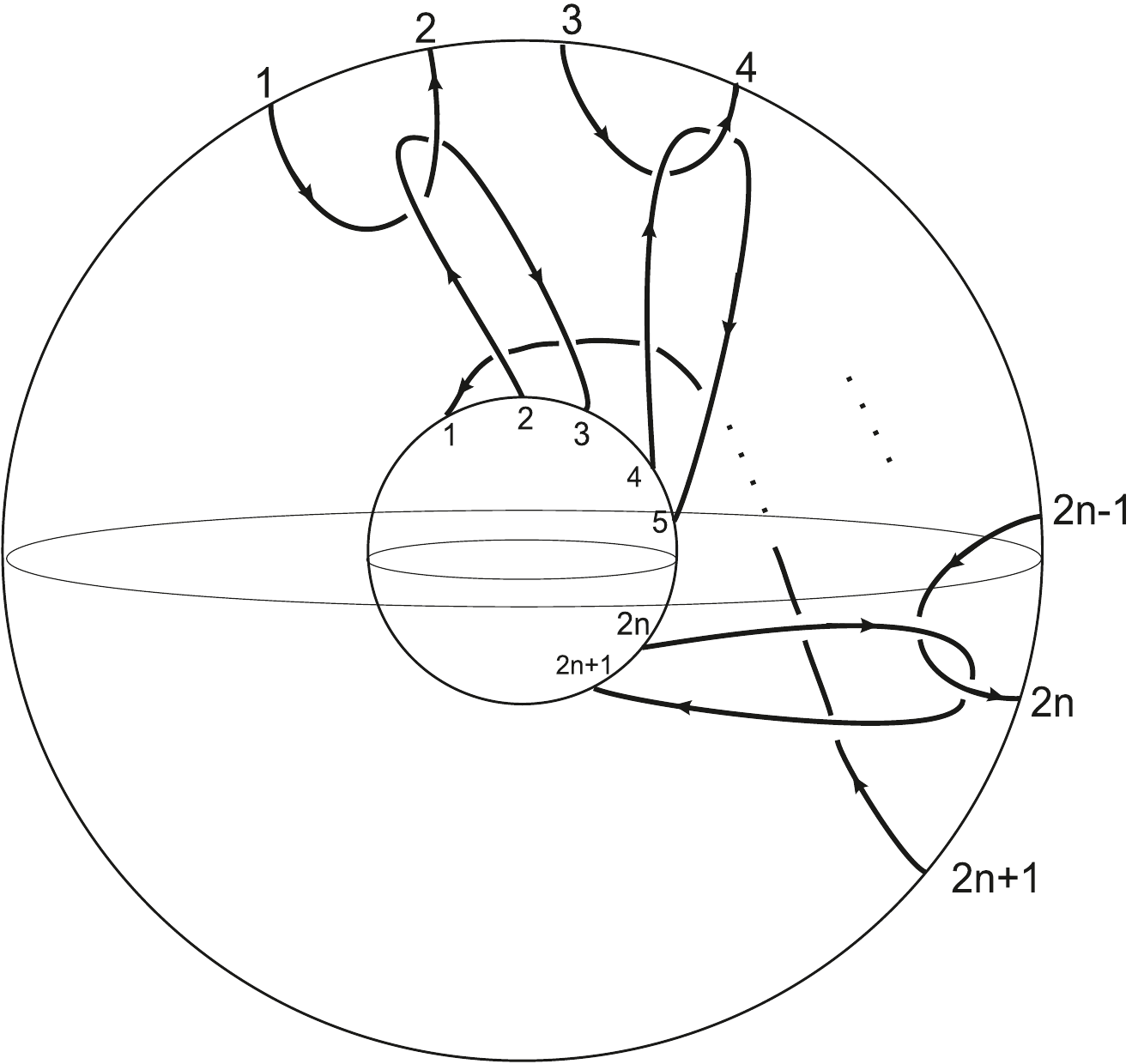}}\caption{\small A generalized Mazur knot $L_{n}$}\label{mn}
\end{figure}
	
According to \cite{GrLauPo2008} a Pixton diffeomorphism $f$ admits an energy Morse function if and only if the knot $L_f$ is {\em trivial} (i.e. equivalent to the standard one). If the knot $L_f$ is not trivial then the number $\rho_{_f}$ of the critical points of a quasi-energy Morse function of $f$ is even as the  classical Morse inequities are true for a continuous Morse function, and hence, $$\rho_{_f}\geqslant 6.$$
The main result of this paper is the proof of Theorem \ref{exact}.
\begin{thm}\label{exact} Let $f$ be a Pixton diffeomorphism ($f\in\mathcal P$) and let the generalized Mazur knot $L_{n},\,n\in\mathbb N$ be its knot. Then the number $\rho_{_f}$ of critical points of a quasi-energy function of $f$ is calculated by\footnote{For $n=1$ Theorem \ref{exact} is proved in \cite{GrLauPo2009}.} $$\rho_{_f}=4+2n.$$
\end{thm}

\section{Construction of Pixton diffeomorphisms}\label{Pico}
In dynamics a wild Artin-Fox arc was for the first time introduced by D.~Pixton in \cite{Pixton1977} where he constructed a Morse-Smale diffeomorphism on the 3-sphere with the unique saddle whose invariant manifolds form an Artin-Fox arc. We give the modern construction of these diffeomorphisms following Ch.~Bonatti and V.~Grines \cite{BoGr} where Pixton diffeomorphisms were also classified (see also \cite{GrMePo2016}, \cite{MePo2018}).
	
For ${\bf x}=(x_1,x_2,x_3)\in\mathbb R^3$ denote $||{\bf x}||=\sqrt{x_1^2+x_2^2+x_3^2}$. Let 
	$h:\mathbb R^3\to\mathbb R^3$ be the diffeomorphism defined by  $h({\bf x})=\frac{{\bf x}}{2}$ and $O(0,0,0)$ be the origin in $\mathbb R^3$. Define the map $p:\mathbb R^3\setminus O\to\mathbb S^{2}\times\mathbb S^1$ by 
	$$p({\bf x})=\left(\frac{{\bf x}}{||{\bf x}||}, 	\log_2(||{\bf x}||)\pmod 1\right).$$
	Let $L\subset (\mathbb S^{2}\times\mathbb S^1)$ be a Hopf knot and let $U(L)$ be its tubular neighborhood. Then the set $\bar L=p^{-1}(L)$ is the $h$-invariant arc in $\mathbb R^3$ 
	and $U(\bar L)=p^{-1}(U(L))$ is its $h$-invariant neighborhood diffeomorphic to $\mathbb{D}^{2}\times\mathbb R^1$ (Fig. \ref{af}). 
\begin{figure}[h!]	\centerline{\includegraphics[width=5 true cm, height=7.5 true cm]{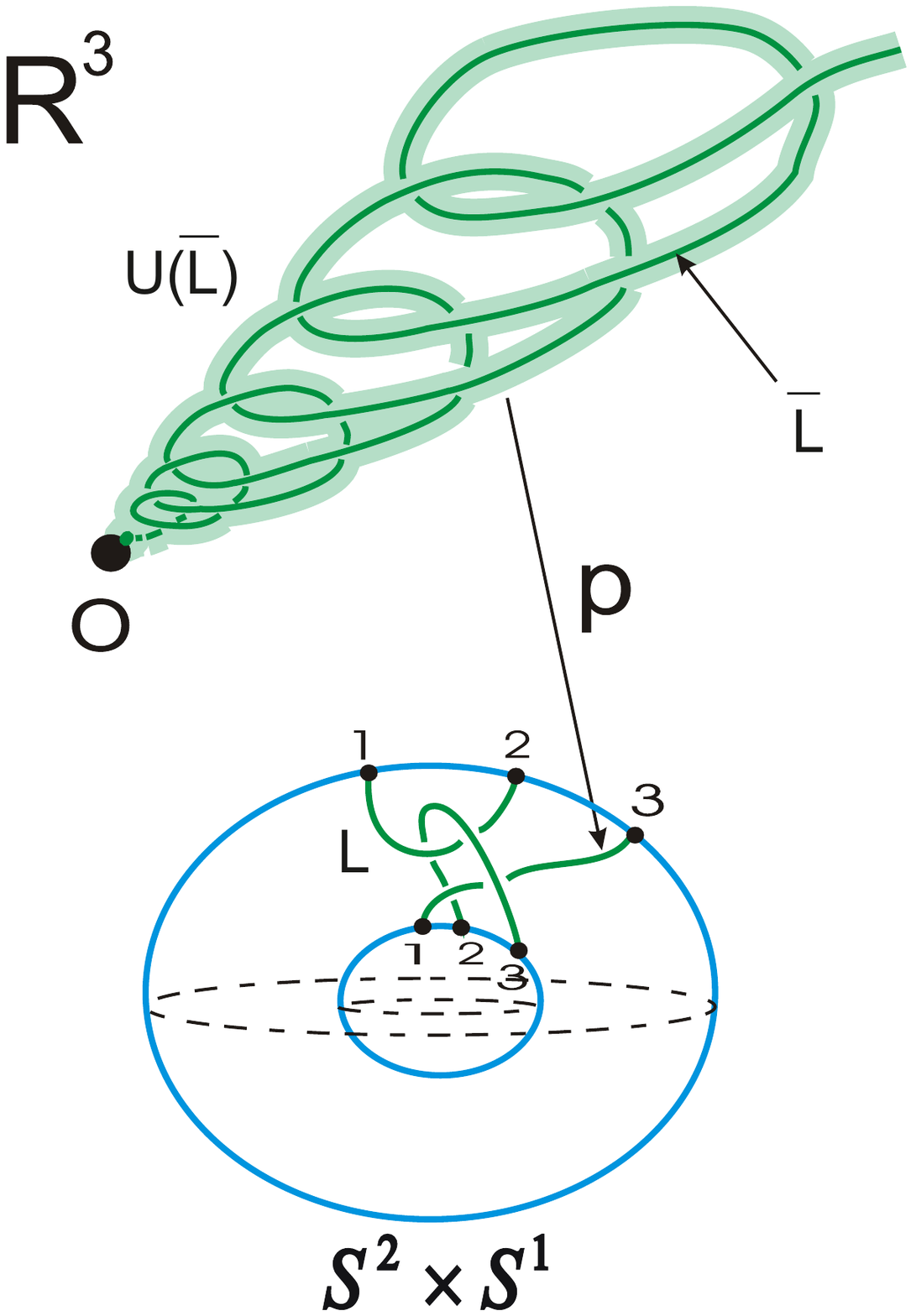}}\caption{\small Suspension of a Hopf knot}\label{af}\end{figure}
	
	Let {$C=\{(x_1,x_2,x_3)\in\mathbb R^3~:~x_2^2+x_{3}^2\leqslant 4\}$} and let $g^t:C\to C$ be the flow  defined by $$g^t(x_1,x_2,x_3)=(x_1+t,x_2,x_3).$$ Then there is a diffeomorphism ${\zeta}:{U(L)}\to C$ that conjugates  $h\vert_{{U(L)}}$ and $g=g^1|_C$. Define the flow $\phi^t$ on $C$ by:  
	$$\begin{cases}
		\dot{x}_1=\begin{cases}1-\frac{1}{9}(x_1^2+x_2^2+x_3^2-4)^2, \quad x_1^2+x_2^2+x_3^2 \leqslant 4 \cr
			1, \quad x_1^2+x_2^2+x_3^2 > 4
		\end{cases}\cr
		\dot{x}_2=\begin{cases}
			\frac{x_2}{2}\big(\sin\big(\frac{\pi}{2}\big(x_1^2+x_2^2+x_3^2-3\big)\big)-1\big), \quad 2<x_1^2+x_2^2+x_3^2\leqslant 4\cr
			-x_2,\quad \quad x_1^2+x_2^2+x_3^2\leqslant 2\cr
			0, \quad x_1^2+x_2^2+x_3^2 > 4
		\end{cases}\cr
		\dot{x}_3=\begin{cases}
			\frac{x_3}{2}\big(\sin\big(\frac{\pi}{2}\big(x_1^2+x_2^2+x_3^2-3\big)\big)-1\big), \quad 2<x_1^2+x_2^2+x_3^2\leqslant 4\cr
			-x_3,\quad \quad x_1^2+x_2^2+x_3^2\leqslant 2\cr
			0, \quad x_1^2+x_2^2+x_3^2 > 4.
		\end{cases}
	\end{cases}$$ 
\begin{figure}[h!]\centerline{\includegraphics	[width=8 true cm, height=5.5 true cm]{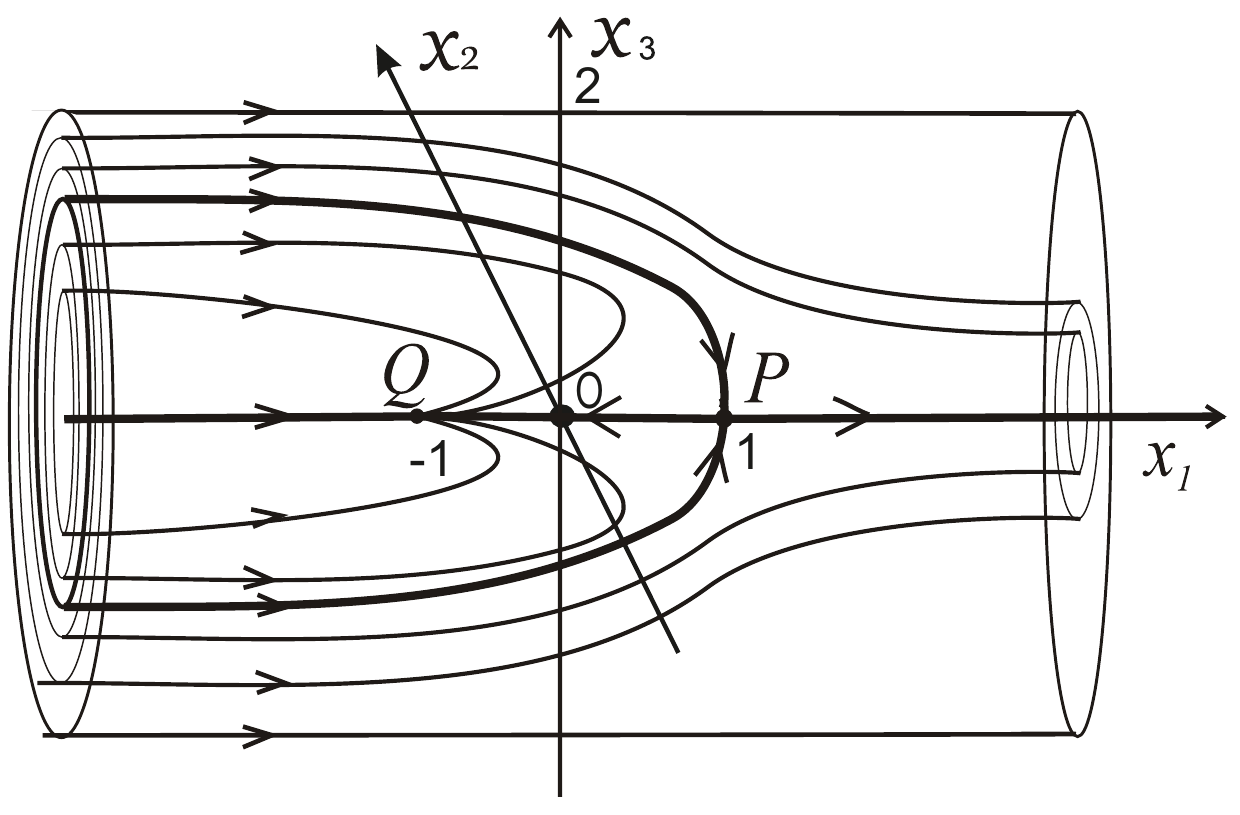}}\caption{\small Trajectories of the flow $\phi^t$}\label{cherry}	\end{figure}
	
	By construction the diffeomorphism $\phi=\phi^1$ has two fixed points: the saddle $P(1,0,0)$ and the sink $Q(-1,0,0)$ (Fig. \ref{cherry}), both being hyperbolic. One unstable separatrice of the saddle $P$ coincides with the open interval  $\left\{(x_1,x_2,x_3)\in\mathbb R^3: \,|x_1|<1,\,x_2=x_3=0\right\}$ in the basin of the sink $Q$ while the other coincides with the ray  $\left\{(x_1,x_2,x_3)\in\mathbb R^3:\,x_1>1,\,x_2=x_3=0\right\}$. 
		Notice that $\phi$ coincides with the diffeomorphism $g=g^1$ outside the ball $\{(x_1,x_2,x_3)\in C:x_1^2+x_2^2+x_3^2\leqslant 4\}$. 
	
Define the diffeomorphism $\bar f_L:\mathbb R^3\to\mathbb R^3$ so that $\bar{f}_{L}$ coincides with the homothety $h$ outside ${U({L})}$ and it coincides with ${\zeta}^{-1}\phi{\zeta}$ on ${U({L})}$. Then $\bar f_{L}$ has in ${U({L})}$ two fixed points: the sink ${\zeta}^{-1}(Q)$ and the saddle ${\zeta}^{-1}(P)$, both being hyperbolic. The unstable separatrice of the saddle  ${\zeta}^{-1}(P)$ lies in ${L}$ (Fig. \ref{af2}). 
\begin{figure}[h!]\centerline{\includegraphics		[width=6 true cm, height=6 true cm]{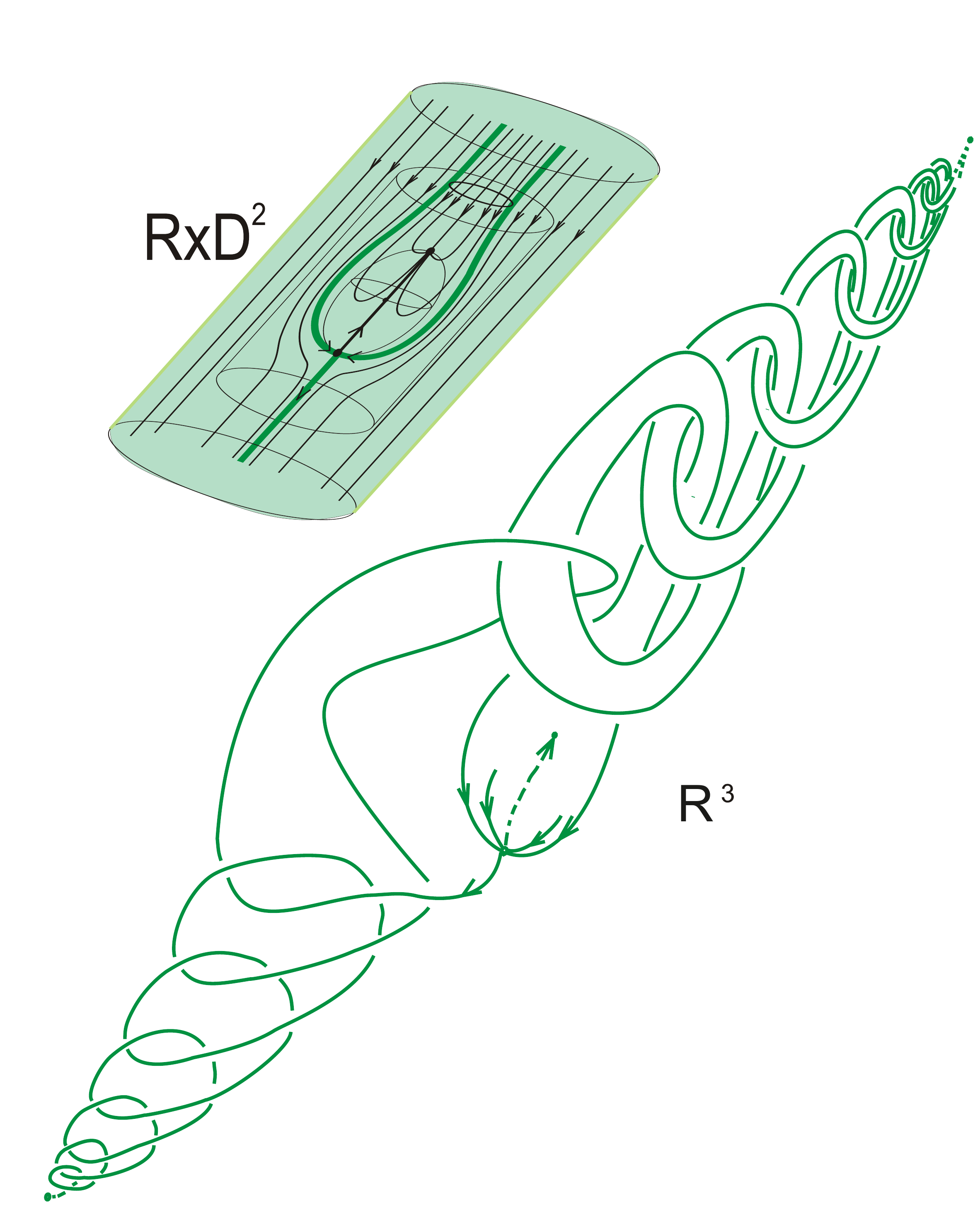}}\caption{\small The phase portrait of the diffeomorphism  $\bar f_L$}\label{af2}\end{figure}
	
	{Now project the dynamics onto the $3$-sphere. Denote by $N(0,0, 0, 1)$ the North Pole of the sphere $\mathbb S^3=\{(x_1,x_2,x_3,x_{4})\in\mathbb R^4:\,x_1^2+x_2^2+x_3^2+x_4^2=1\}$. For every point $(x_1,x_2,x_3,x_{4})\in(\mathbb{S}^3\setminus\{N\})$ there is the unique line in $\mathbb R^{4}$ passing through this point and the point $N$. This line intersects $\mathbb R^3\subset\mathbb R^4$ at exactly one point $\vartheta_+(x_1,x_2,x_3,x_{4})$ (Fig. \ref{stereo}). One can easily check that $$\vartheta_+(x_1,x_2,x_3,x_{4})=\left(\frac{x_1}{1-x_{4}},
		\frac{x_{2}}{1-x_{4}},\frac{x_3}{1-x_{4}}\right).$$ The diffeomorphism $\vartheta_+:\mathbb S^3\setminus\{N\}\to\mathbb R^3$ is called a {\em stereographic projection}.
\begin{figure}[h!]\centering{\includegraphics[width=8cm,height=4.5 cm]{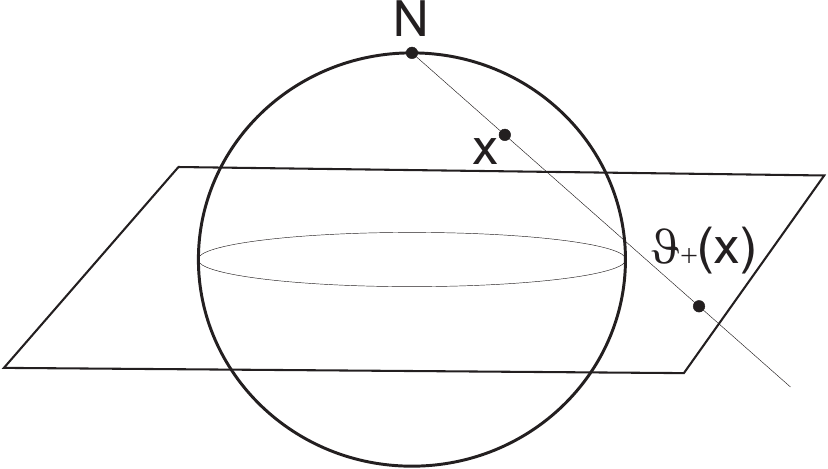}} \caption{\small The stereographic projection.}\label{stereo}	\end{figure}
	
	{By construction $\bar{f}_{L}$ coincides with the homothety $h$ in some neighborhood of the point $O$ and in some neighborhood of the infinity. Therefore, it induces on $\mathbb{S}^3$ the Morse-Smale diffeomorphism $$f_{{L}}(x)=\begin{cases}\vartheta_+^{-1}(\bar f_{L}(\vartheta_+(x))),~{x}\neq N;\cr N,~{x}=N\end{cases}.$$} It follows directly from the construction that the non-wandering set of $f_{{L}}$ consists of exactly four fixed hyperbolic points: two sinks $\omega=\vartheta_+^{-1}({\zeta}^{-1}(Q))$, $S$, one saddle $\sigma=\vartheta_+^{-1}({\zeta}^{-1}(P))$ and one source $N$.
	We say the constructed diffeomorphism to be {\em model} and it is of Pixton class.
	\begin{proposition}[Theorem 1, Theorem 2, Theorem 3 of \cite{BoGr}] $ $
		\begin{itemize}
			\item Any diffeomorphism $f\in\mathcal P$ is topologically conjugate to some model diffeomorphism $f_L$.
			\item Two model diffeomorphisms $f_L,f_{L'}$ are topologically conjugate if and only if their knots $L,L'$ are equivalent.
		\end{itemize} \label{Pidy}
	\end{proposition}

\begin{remark} Since $\rho_{_f}$ is a topological invariant, it suffices to prove Theorem~\ref{exact} only for the model diffeomorphisms $f_L\in G$.
\end{remark}
	
\section{Genus of Hopf knot}\label{appr}
In this section we introduce the notion of genus for a Hopf knot and use it to estimate the number of critical points of the quasi-energy function of the Pixton diffeomorphism defined by this knot.

Let $L$ be a Hopf knot. We say a closed orientable surface $\Sigma\subset\mathbb S^2\times\mathbb S^1$ to be a {\em secant surface of the knot $L$} if it intersects $L$ at a unique point. The minimally possible genus $g_{_L}$ of the secant surface is called the {\it genus of the knot} $L$. The secant surface of $L$ of genus $g_{_L}$ is said to be {\it minimal}.

Since the secant surface $\Sigma$ in $\mathbb S^2\times\mathbb S^1$ cuts a Hopf knot transversely at exactly one point the surface $\Sigma$ is homologous to the fiber $\mathbb S^2\times\{x\}$. 
Furthermore, its lift by $p$ on the universal cover $\mathbb R^3\setminus O$ consists of connected components parametrized by $\mathbb Z$, so that $p$ is a diffeomorphism from such a  component $\bar \Sigma$ onto
$\Sigma$. This implies that $\bar \Sigma$ bounds a trapping neighborhood $Q_\Sigma$ of $O$, that is $h(Q_\Sigma)\subset int\,Q_\Sigma$.

Let $\bar L=p^{-1}(L)$ be the  cover of $L$ in $\mathbb R^3\setminus O$ and $\bar y=Q_\Sigma\cap\bar L$. Recall that the surface $\bar\Sigma\setminus\bar y$ is called {\em incompressible in $\mathbb R^3\setminus(O\cup\bar L)$} if the fact that a simple closed curve $c\subset int~(\bar\Sigma\setminus\bar y)$ is contractible on $\bar\Sigma\setminus\bar y$ means that it bounds a smoothly embedded 2-disk $D\subset int~(\mathbb R^3\setminus(O\cup\bar L))$ such that $D\cap (\bar\Sigma\setminus\bar y)=\partial D=c$.

\begin{lemma}\label{compr} If  $\Sigma$ is a minimal secant surface of the knot $L$ then the surface $\bar\Sigma\setminus\bar y$ is incompressible in $\mathbb R^3\setminus(O\cup\bar L)$.
\end{lemma}
\begin{proof} Let $\Sigma$ be a minimal secant surface of $L$ and let $\bar y$ be the unique point of the intersection $\bar L\cap \bar\Sigma$. Assume the opposite: the surface $\bar\Sigma\setminus\bar y$ is compressible in $\mathbb R^3\setminus(O\cup\bar L)$. Then there is a non-contractible simple closed curve $c\subset int~(\bar\Sigma\setminus\bar y)$ and there is the smoothly embedded 2-disk $D\subset int~(\mathbb R^3\setminus(O\cup\bar L))$ such that $D\cap (\bar\Sigma\setminus\bar y)=\partial D=c$ (see, for example, \cite{Neu}).
Then we have two possibilities: 
\begin{equation}\label{(1)}
(int\,D)\cap\left(\bigcup\limits_{k\in\mathbb Z}h^k(\bar\Sigma)\right)=\emptyset,
\end{equation}  
\begin{equation}\label{(2)}
(int\,D)\cap\left(\bigcup\limits_{k\in\mathbb Z}h^k(\bar\Sigma)\right)\neq \emptyset. 
\end{equation} 
In case (1) two subcases are possible: (1a) 
$D\subset Q_\Sigma$, (1b) $D\subset (\mathbb R^3\setminus int\,Q_\Sigma)$. For case 1a) let $N(D)\subset Q_\Sigma$ be a tubular neighborhood of the disk $D$. Then exactly one connected component of the set $Q_\Sigma\setminus int\,N(D)$ intersects $\bar L$. According to (\ref{(1)}) this component is a trapping neighborhood of $O$ and its boundary intersects $\bar L$ at a unique point. The projection of this boundary into $\mathbb S^2\times\mathbb S^1$ is, therefore, the secant surface of $L$ of genus less than $g_L$. This contradicts the fact that the surface $\Sigma$ is minimal. In case 1b) let $N(D)\subset (\mathbb R^3\setminus int\,Q_\Sigma)$ be a tubular neighborhood of $D$. Then due to (\ref{(1)}) the set $Q_\Sigma\cup N(D)$ is a trapping neighborhood of $O$ and its boundary intersects $\bar L$ at a unique point. The projection of this boundary into $\mathbb S^2\times\mathbb S^1$ is, therefore, the secant surface of $L$ of genus less than $g_L$ and we have the same contradiction.

In case (2) without loss of generality assume the intersection $int\,D\cap(\bigcup\limits_{k\in\mathbb Z}h^k(\bar\Sigma))$ to be transversal and denote it by $\Gamma$. Let $\gamma$ be a curve from $\Gamma$. We say the curve $\gamma$ to be {\it innermost} if it is the boundary of the disk $D_\gamma\subset D$ such that $int\,D_\gamma$ contains no curves of $\Gamma$. Consider this innermost curve $\gamma\subset f^k(\Sigma)$. There are two subcases: a) $\gamma$ is essential on $f^k(\Sigma)$ and b) $\gamma$ is contractible on  $f^k(\Sigma)$. In case a) the arguments of the case (1) apply for the body $f^k(Q_\Sigma)$ and the disk $D_\gamma$ and we get the contradiction to the minimality of the surface $\Sigma$. In case b) denote by $d_\gamma\subset f^k(\Sigma)$ the 2-disk bounded by $\gamma$ and denote by $B_\gamma\subset(\mathbb R^3\setminus O)$ the 3-ball bounded by the 2-sphere $D_\gamma\cup d_\gamma$. Consider: b1) $B_\gamma\subset f^k(Q_\Sigma)$ and b2) $B_\gamma\subset (\mathbb R^3\setminus int\,f^k(Q_\Sigma))$. For b1) let $N(B_\gamma)\subset f^k(Q_\Sigma)$ be a tubular neighborhood of $B_\gamma$. { Then the set $Q_\Sigma\setminus int\, N(B_\gamma)$ is a trapping neighborhood of $O$ because the curve $\gamma$ lies in its interior and the boundary of $Q_\Sigma\setminus int\, N(B_\gamma)$ intersects $\bar L$ at a unique point.} The projection of this boundary into $\mathbb S^2\times\mathbb S^1$ is, therefore, the secant surface of the knot $L$ of genus $g_L$ for which the number of connected components of the set $\Gamma$  is less. We get the same result for b2) for the set $Q_\Sigma\cup N(B_\gamma)$. Thus, iterating the process we come either to the case a) or to the case (1) and get a contradiction.
\end{proof}

Now, let us consider a model diffeomorphism $f_L\in G$.  Denote by $\ell$ the unstable separatrice of the saddle $\sigma$ lying in the basin of the sink $S$. Since the diffeomorphism $f_L$ is conjugate with the homothety $h$ in some neighborhood $V_S$ of $S$ by means of the stereographic projection $\vartheta_+$ then a natural projection $$p_S:W^s_S\setminus S\to\mathbb S^2\times\mathbb S^1$$ is correctly defined by the formula $$p_S(w)=p(\vartheta_+(f^{k_w}(w))),\,w\in(W^s_S\setminus S),\,f^{k_w}(w)\in V_S.$$ Moreover, $p_S(\ell)=L$.

\begin{lemma} Any Morse-Lyapunov function $\varphi:\mathbb S^3\to\mathbb R$ of the diffeomorphism $f_L$ has a connected component $\bar\Sigma_1\subset (W^s_S\setminus S)$ of its level set such that $p_S(\bar\Sigma_1)$ is a secant surface for $L$. \label{genus1}
\end{lemma}
\begin{proof} Consider an arbitrary Morse-Lyapunov function $\varphi:\mathbb S^3\to\mathbb R$ of the diffeomorphism $f_L$. To be definite let $\varphi(S)=0$, $\varphi(\sigma)=1$ and $\varphi(N)=3$. From the definition of the Morse-Lyapunov function it follows that $\varphi|_{\ell}$ monotonically decreases in some neighborhood of the saddle $\sigma$. Therefore, there is $\varepsilon_1\in(0,1)$ such that the interval  $(1-\varepsilon_1,1)$ contains no critical values of $\varphi$ and the connected component $\bar\Sigma_{1}$ of the level set $\varphi^{-1}(1-\varepsilon_1)$ intersects the separatrice $\ell$ at the unique point. Denote this point by $w_1$. 

Let $\bar Q_1$ be the connected component of the set $\varphi^{-1}([0,1-\varepsilon_1])$ which contains the segment $[w_1,S]$ of the closure of the separatrice $\ell$. Since $\varphi$ decreases along the trajectories of $f$, the values of $\varphi$ on $W^s_\sigma$ are greater than 1. Therefore, the manifold $\bar Q_1$ lies in the manifold $W^s_S$. 
\end{proof}

\begin{lemma} Let $\varphi$ and $\bar\Sigma_1$ be the function and the surface from Lemma \ref{genus1}. Let $g$ be the genus of  $\bar\Sigma_1$. Then $\varphi$ has at least $4+2g$ critical points. \label{genus2}
\end{lemma}
\begin{proof} Let the function $\varphi|_{\bar Q_1}$ have $k_q,\,q\in\{0,\dots,3\}$ critical points of index $q$. Due to \cite[Theorem 6.1]{Fomad} on the manifold $\bar Q_1$ there exists a self-indexing Morse function $\psi$ (the value of the function in a critical point equals the index of this point) which has $k_q$ critical points of index $q$ and which is constant on $\partial{\bar Q_1}$. Thus, the manifold $\bar Q_1$ is the surface $\tilde Q_1$ of genus $g_1=1+k_1-k_0$ with attached handles of indexes 2 and 3. Then the genus of any surface of the set $\partial\bar Q_1$ cannot be greater than $g_1$, that is $$g_1\geqslant g.$$
On the other hand, the number of critical points of  $\varphi|_{\bar Q_1}$ is not less than $k_0+k_1$. If  $k_0\geqslant 1$ and $g_1=1+k_1-k_0$ then one gets $k_0+k_1=g_1+2k_0-1\geqslant g_1+1$. Thus, $\varphi|_{\bar Q_1}$ has at least $g+1$ critical points.

Analogously, there is $\varepsilon_2\in(0,1)$ for which the interval $(1,1+\varepsilon_2)$ contains no critical points of  $\varphi$ and  the connected component $\bar Q_{2}$ of the level set $\varphi^{-1}([0,1+\varepsilon_2)]$ contains $cl(W^u_\sigma)$ in its interior while the intersection $\bar Q_{2}$ with $W^s_\sigma$ is the unique 2-disk. Due to construction the function $\varphi|_{\bar Q_2}$ has at least $g_1+3$ critical points and genus of the connected components of $\partial\bar Q_2$ is less or equals $g_1$.  
Denote by $\bar\Sigma_2$ the connected component of $\partial\bar Q_2$ which intersects $W^s_\sigma$ and denote by $g_2$ its genus. The surface $\bar\Sigma_2$ divides the manifold $W^u_N$ into two parts, one of which $Q_2$ is a trapping neighborhood of $N$ with respect to $f_L^{-1}$. Arguing as above one comes to conclusion that the number of critical points of  $\varphi|_{Q_2}$ is at least  $g_2+1$. Therefore, the total number of critical points of $\varphi$ is greater or equal to $$g_1+3+g_2+1\geqslant 4+2g_1\geqslant 4+2g.$$
\end{proof}
From Lemma \ref{genus1} and Lemma \ref{genus2} we get the following fact.
\begin{corollary}\label{geny} For any model diffeomorphism $f_L\in\mathcal P$ the following estimation holds
\begin{equation}\label{gen}
\rho_{_{f_L}}\geqslant 4+2g_{_{L}}.
\end{equation}
\end{corollary}

\section{The generalized Mazur knot $L_{n}$}
In this section we show that the genus $g_{_{L_{n}}}$ of a generalized Mazur knot equals ${n}$. At first we give a detailed description of construction of $L_n$.
\subsection{Construction of the generalized Mazur knot  $L_n$}
Recall that for ${\bf x}=(x_1,x_2,x_3)\in\mathbb R^3$ and $||{\bf x}||=\sqrt{x_1^2+x_2^2+x_3^2}$ we define the homothety $h:\mathbb R^3\to\mathbb R^3$ by the formula $$h({\bf x})=\frac{{\bf x}}{2}.$$ For $O(0,0,0)\in\mathbb R^3$ we define the natural projection $p:\mathbb R^3\setminus O\to\mathbb S^{2}\times\mathbb S^1$ by the formula 
	$$p({\bf x})=\left(\frac{{\bf x}}{||{\bf x}||}, 	\log_2(||{\bf x}||)\pmod 1\right).$$
 Consider the set  $$K_0=\left\{{\bf x}\in\mathbb R^3:\frac12\leq ||{\bf x}||< 1\right\}$$ whose closure $K=cl\,K_0$ is a fundamental domain of $h|_{\mathbb R^3\setminus O}$ with the boundary $$\mathbb S^2=\left\{{\bf x}\in\mathbb R^3:\,||{\bf x}||= 1\right\},\,h(\mathbb S^2).$$ 
Pick on the circle $$\mathbb S^1=\left\{(x_1,x_2,x_3)\in \mathbb R^3: x_1^2+x_2^2=1,\,x_3=0\right\}$$ pairwise distinct points $\alpha_1,\dots,\alpha_{2n+1}$ numbered in counter-clockwise order (Fig.~\ref{wil}). Let $a_i,\,i\in\{1,\dots,2n\}$ be the arc of the circle $\mathbb S^1$ bounded by $\alpha_i$, $\alpha_{i+1}$ whose interior contains no points of  $\{\alpha_1,\dots,\alpha_{2n+1}\}$. Let  $B,\,A_{i}\subset cl\,K_0,\,i\in\{1,\dots,2n\}$ be pairwise disjoint smooth arcs such that:
\begin{enumerate}
\item[1)] the boundary points of $B$ are $\alpha_{2n+1}$, $h(\alpha_1)$;  the boundary points of $A_{2j-1}$ are $\alpha_{2j-1}$, $\alpha_{2j}$; the boundary points of $A_{2j}$ are $h(\alpha_{2j})$, $h(\alpha_{2j+1})$ for $j\in\{1,\dots,n\}$ and the arc $cl (h(A_1)\cup A_2\cup\dots\cup h(A_{2n-1})\cup A_{2n}\cup B)$ is smooth; 
\item[2)] the closed curves $c_{2j-1}=cl(a_{2j-1}\cup A_{2j-1}),\,c_{2j}=cl(h(a_{2j})\cup A_{2j})$ bound in $cl\,K_0$ the 2-disks $d_{2j-1},\,d_{2j}$ whose transversal intersection is  the arc $l_j$ with the boundary points $b_{2j-1}=d_{2j-1}\cap A_{2j}$, $b_{2j}=d_{2j}\cap A_{2j-1}$ and $cl\,K_0\setminus(B\cup\bigcup\limits_{i=1}^{2n}d_i)$ is homeomorphic to $D_n\times[0,1]$, where $D_n$ is the open 2-disc with $n$ holes.
\end{enumerate}
\begin{figure}[h]\center{\includegraphics[width=
0.7\linewidth]{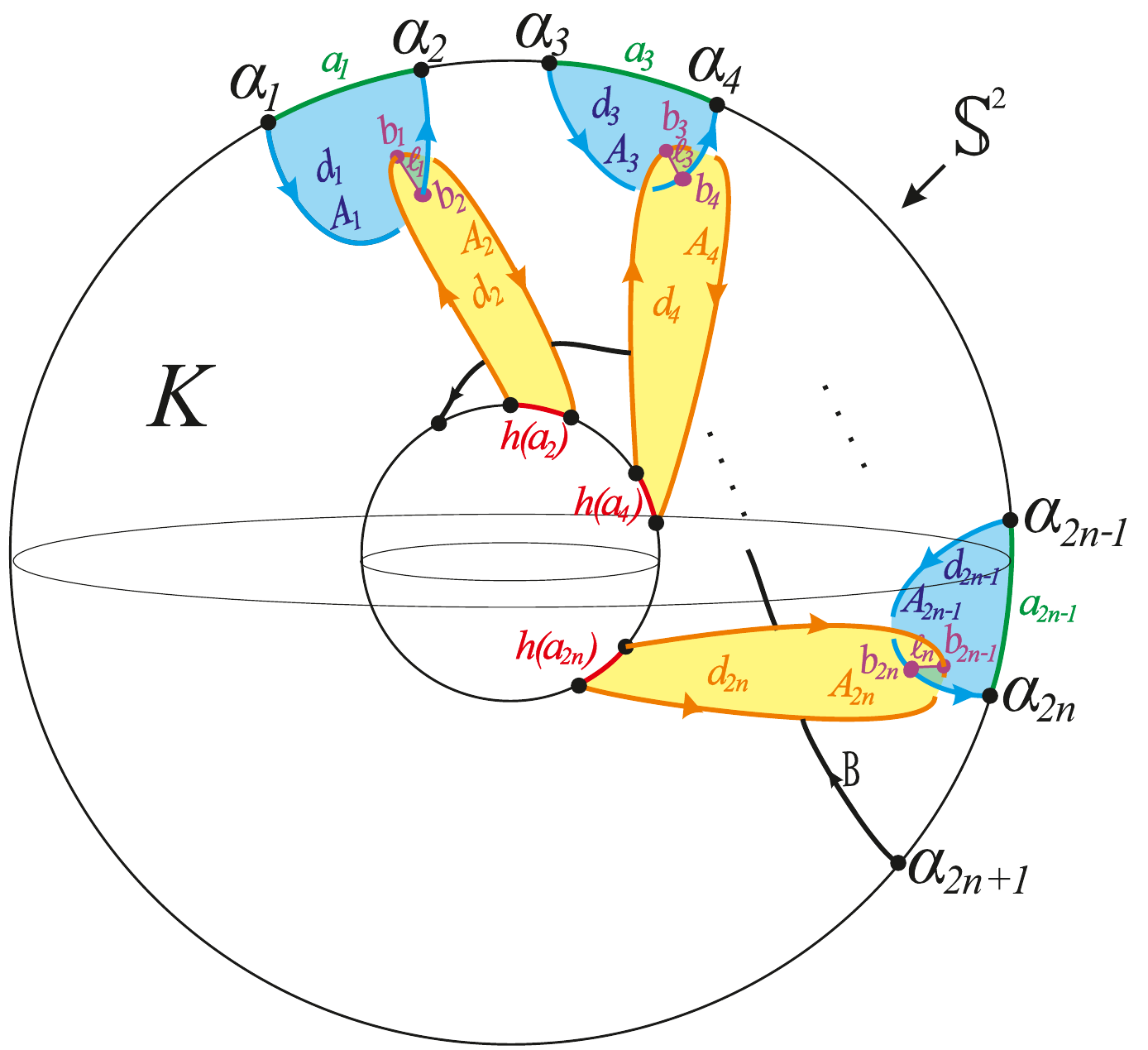}}\caption{\small Construction of the knot $L_{n}$}\label{wil}
\end{figure}
Let $$\bar L_n=\bigcup\limits_{k\in\mathbb Z}h^k(B\cup A_1\cup\dots\cup A_{2n}),\,L_n = p (\bar L_n).$$ 

\subsection{The genus of the knot $L_n$}
\begin{lemma}\label{rod1} The  knot $L_n$ admits a minimal secant $\Sigma$ such that $\bar\Sigma\subset K_0$ and $\bar L_n\cap\bar\Sigma=h(\alpha_1)$.
\end{lemma}
\begin{proof} Let $\Sigma_0$ be some minimal secant surface of $L_n$. Then there exists the connected component $\bar\Sigma_0$ of $p^{-1}(\Sigma_0)$ such that it intersects the curve $\bar L_n$ at the point $\bar y_0$ situated on $\bar L_n$ between $\alpha_1,h(\alpha_1)$ and that bounds a trapping neighborhood  $Q_{\Sigma_0}$ of $O$. Without loss of generality let $\bar y_0=h(\alpha_1)$ (otherwise the desired surface is constructed by removing the tubular neighborhood of the arc $[\bar y_0, h(\alpha_1)]\subset\bar L_n$ from $Q_{\Sigma_0}$).

Denote by $k_+,k_-\geqslant 0$ the maximal integers for which $f^{k}(K_0)\cap\bar\Sigma_0\neq\emptyset,\,f^{-k}(K_0)\cap\bar\Sigma_0\neq\emptyset,\,k\geqslant 0$, respectively. If $k_+=k_-=0$ then $\bar\Sigma_0$ is the desired surface. Otherwise we show the way to decrease by 1 the number
$k_+>0$ (for $k_-$ the arguments are the same) using isotopy of the secant surface.

Notice that $\bar\Sigma_0\cap f^{k_+}(c_{2j-1})=\emptyset,\,j\in\{1,\dots,n\}$. Without loss of generality let the intersection $\Gamma=\bigcup\limits_{j=1}^nf^{k_+}(d_{2j-1})\cap\bar\Sigma_0$ be transversal. Let $\gamma$ be a curve from  $\Gamma$. Then $\gamma$ bounds the unique disk $D_\gamma\subset f^{k_+}(d_{2j-1})$.
There are two possibilities: 1) $f^{k_+}(b_{2j-1})\notin D_\gamma$, 2) $f^{k_+}(b_{2j-1})\in D_\gamma$. In case 1) we say the curve $\gamma$ to be {\em innermost} if it bounds the disk $D_\gamma\subset f^{k_+}(d_{2j-1})$ such that $int\,D_\gamma$ contains no curves of  $\Gamma$. Consider this innermost curve $\gamma$. Due to Lemma \ref{compr} the surface $\bar\Sigma_0\setminus\bar y_0$ is incompressible in $\mathbb R^3\setminus(O\cup\bar L_n)$ and, therefore, there exists the disk  $d_\gamma\subset(\bar\Sigma_0\setminus\bar y_0)$ bounded by $\gamma$. Denote by $B_\gamma\subset(\mathbb R^3\setminus (O\cup\bar L_n))$ the 3-ball bounded by the 2-sphere $D_\gamma\cup d_\gamma$. Consider two subcases: 1a) $B_\gamma\subset Q_{\Sigma_0}$ and 1b) $B_\gamma\subset (\mathbb R^3\setminus int\,Q_{\Sigma_0})$. 

In case 1a) let $N(B_\gamma)\subset Q_{\Sigma_0}$ be a tubular neighborhood of the ball $B_\gamma$. Then the set $Q_\Sigma\setminus int\, N(B_\gamma)$ is a trapping neighborhood of $O$ because the curve $\gamma$ {lies in its interior} and its boundary intersects $\bar L_n$ at a unique point. The projection of this boundary to $\mathbb S^2\times\mathbb S^1$ is, therefore, a secant surface of $L_n$ of the same genus as $\Sigma_0$. For it the number of the connected components of the set $\Gamma$ is less. One gets the same result in  case 1b) for the set $Q_{\Sigma_0}\cup N(B_\gamma)$. 

If we continue this process then we get the secant surface of $L_n$ of the same genus as $\Sigma_0$ and for which the set $\Gamma$ contains no curves of type 1). Denote the resulting surface again by $\Sigma_0$. Now the set $\Gamma$ consists only of the curves $\gamma$ bounding the disk $D_\gamma\subset f^{k_+}(d_{2j-1})$ which contains the point $f^{k_+}(b_{2j-1})$. Since $f^{k_+}(b_{2j-1}\sqcup c_{2j-1})\subset (\mathbb R^3\setminus Q_{\Sigma_0})$, the number of these curves on the disk $f^{k_+}(d_{2j-1})$ is even. Since the surface  $\bar\Sigma_0\setminus\bar y_0$ is incompressible in $\mathbb R^3\setminus(O\cup\bar L_n)$, all these curves are pairwise homotopic on $\bar\Sigma_0\setminus\bar y_0$ and, therefore, they lie in the annulus $\kappa\subset(\bar\Sigma_0\setminus\bar y_0)$ bounded by the pair of these curves $\gamma_1,\gamma_2$. Denote by $\tilde\kappa\subset f^{k_+}(d_{2j-1})$ the annulus bounded by the same curves on the disk $f^{k_+}(d_{2j-1})$. Let $\tilde\Sigma_0=\bar\Sigma_0\setminus\kappa\cup\tilde\kappa$. Due to construction the surface $\tilde\Sigma_0$ is of the same genus as the surface $\bar\Sigma_0$ and it bounds a trapping neighborhood of $O$. Having removed a tubular neighborhood of the annulus $\tilde\kappa$ from this body we get a trapping neighborhood of $O$ whose boundary does not intersect the disk  $f^{k_+}(d_{2j-1})$ and whose projection to $\mathbb S^2\times\mathbb S^1$ is the secant surface of the knot $L_n$ of the same genus as $\Sigma_0$.     

If we continue this process then we get a secant surface of $L_n$ of the same genus as $\Sigma_0$ and for which the set $\Gamma$ is empty. Denote this surface again by $\Sigma_0$. Without loss of generality let the intersections of the surface $\bar\Sigma_0$ with the spheres $f^k(\mathbb S^2)$ be transversal. Denote by $\mathcal F$ the set of the connected components of the intersection $f^{k_+}(K_0)\cap\bar\Sigma_0$. Now we show the way to reduce by 1 the number of the components in $\mathcal F$ using isotopy of the secant surface.

Denote by $Q$ the set obtained by removal from the annulus $f^{k_+}(K_0)$ of the tubular neighborhoods of the disks $f^{k_+}(d_{2j-1})$ as well as the tubular neighborhoods of the curves $f^{k_+}(A_{2j}),\,j\in\{1,\dots,n\}$, avoiding $\bar\Sigma_0$. Then $Q$ is homeomorphic to the direct product of the 2-sphere with  $2n+1$ deleted points and the segment. 
Since $Q\cap\bar\Sigma_0=f^{k_+}(K_0)\cap\bar\Sigma_0$ and since $\bar\Sigma_0\setminus\bar y_0$ is incompressible in $\mathbb R^3\setminus(O\cup\bar L_n)$, each connected component of  $F\in\mathcal F$ is incompressible in $Q$. Due to \cite[Corollary 3.2]{Wald} there exists a surface $\tilde F\subset f^{k_+-1}(\mathbb S^2)$ diffeomorphic to $F$ for which $\partial F=\partial{\tilde F}$ and the surface $F\cup\tilde F$ bounds in $Q$ the body $\Delta$ diffeomorphic to the direct product $F\times[0,1]$. Then we replace the part $F$ of $\bar\Sigma_0$ with $\tilde F$. If we continue the process we get the desired secant surface  $\Sigma\subset K_0$. 
\end{proof}

\begin{lemma}\label{rod} The genus $g_{_{L_n}}$ of the knot $L_n$ equals $n$.
\end{lemma}
\begin{proof} Since there is a secant surface of $L_n$ of genus $n$  (Fig.~\ref{mnt}), we have $g_{_{L_n}}\leqslant n$. Now we show that $g_{_{L_n}}\geqslant n$. 
\begin{figure}[h]\center{\includegraphics[width=
0.7\linewidth]{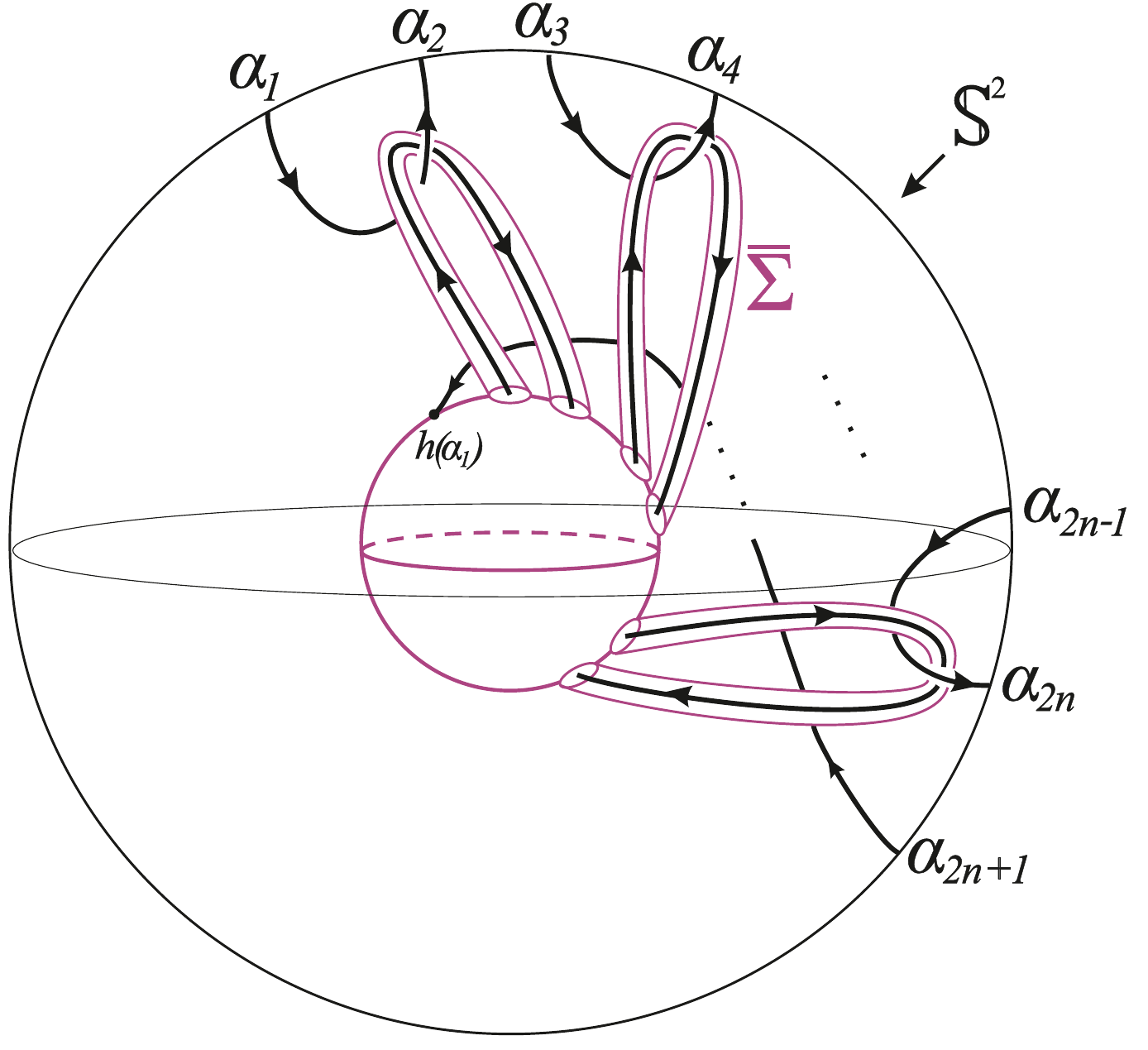}}\caption{\small A secant surface of $L_{n}$ of genus $n$}\label{mnt}
\end{figure}

By Lemma \ref{rod1} there exists a minimal secant surface $\Sigma$ for $L_n$ such that $\bar\Sigma\subset K$ and $\bar L_n\cap\bar\Sigma=h(\alpha_1)$.
{Notice (see, for instance, \cite[Exercise 2.8.1]{Dav}) that the fundamental group  $\pi_1(K\setminus\bar L_n)$ has $2n$ generators $\gamma_1,\dots,\gamma_{2n}$, each of which $\gamma_i$, $i\in\{1,\dots,2n\}$ being the generator of the punctured disk $d_i\setminus b_i$   (Fig. \ref{gene}).
\begin{figure}[h]\center{\includegraphics[width=
0.7\linewidth]{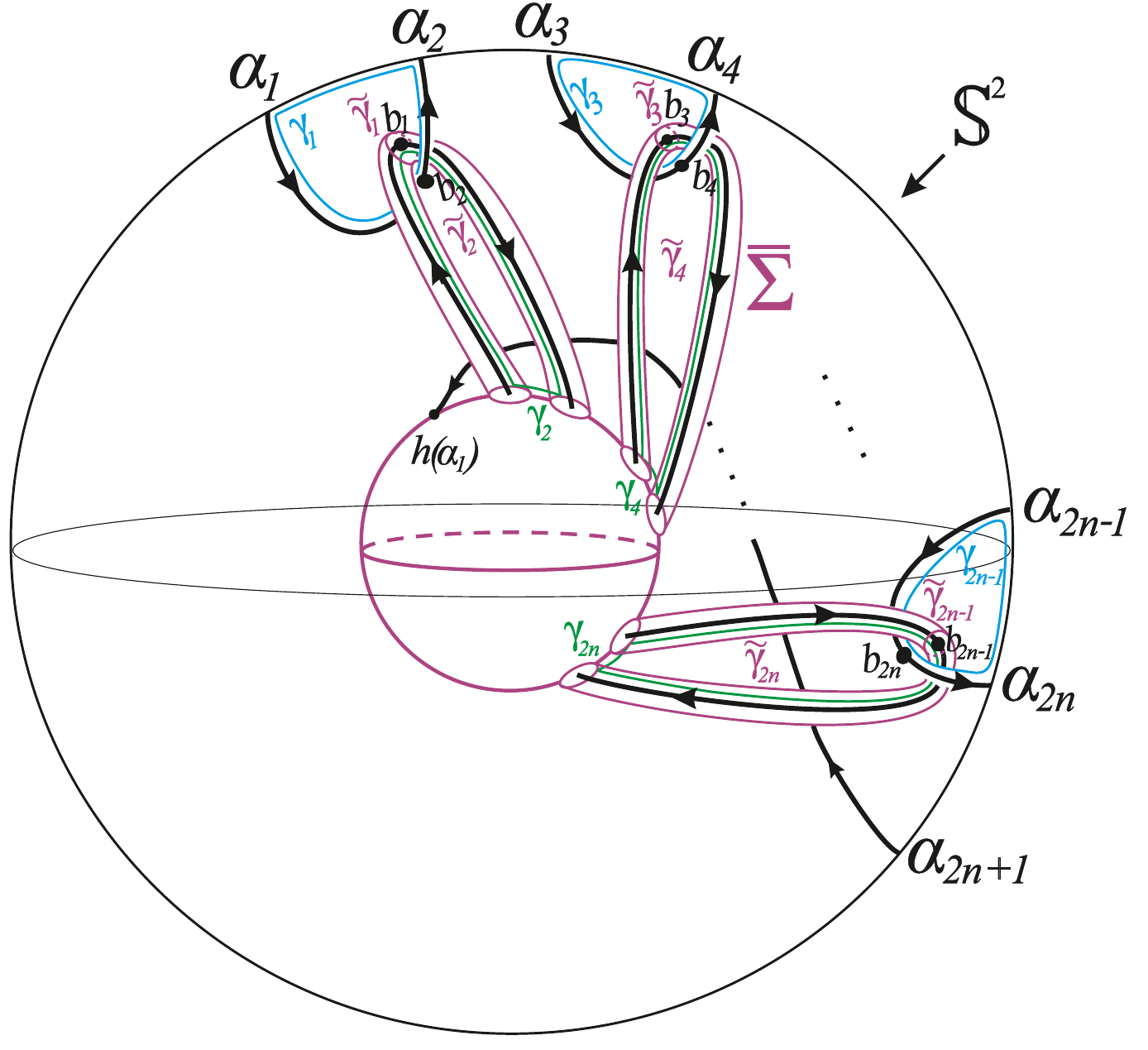}}\caption{\small Generators of the group $\pi_1(K\setminus\bar L_n)$}\label{gene}
\end{figure} Since $b_{2j-1}\in int\, Q_\Sigma$ and $c_{2j-1}\cap Q_\Sigma=\emptyset$, there exists the connected component of $\tilde d_{2j-1}$ of the intersection $d_{2j-1}\cap Q_\Sigma$ which contains the point $b_{2j-1}$.} This component is the 2-disk bounded by the curve  $\tilde\gamma_{2j-1}\subset(\bar\Sigma\setminus h(\alpha_1))$ with holes and the curves $\gamma_{2j-1},\,\tilde\gamma_{2j-1}$ are homotopic on the punctured disk $d_{2j-1}\setminus b_{2j-1}$. In the same way one finds the curves $\tilde\gamma_{2j}\subset(\bar\Sigma\setminus h(\alpha_1))$ homotopic to the curves $\gamma_{2j}$ on the punctured disk $d_{2j}\setminus b_{2j}$  (Fig. \ref{gene}). Due to Lemma \ref{compr} the surface $\bar\Sigma\setminus h(\alpha_1)$ is incompressible in $K\setminus\bar L_n$. Then the curves $\tilde\gamma_1,\dots,\tilde\gamma_{2n}$ are pairwise non-homotopic to the generators on the surface $\bar\Sigma\setminus h(\alpha_1)$. Therefore, the genus of the surface $\bar\Sigma$ cannot be less than $n$.
\end{proof}

\section{Construction of a quasi-energy function for a Pixton diffeomorphism with the Hopf knot  $L_n$}
	
Let $f$ be a Pixton diffeomorphism constructed for a generalized Mazur knot $L_n$. Then its non-wandering set $\Omega_f$ consists of four points: two sinks $\omega,\,S$, a source $N$ and a saddle $\sigma$. Then $W^u_\sigma\setminus\sigma$ consists of two separatrices $\ell_\omega,\,\ell_S$ respective closures of which contain the sinks $\omega,\,S$, the separatrice $\ell_\omega$ being tame while $\ell_S$ being wild. Let $\bar\Sigma$ be the surface of genus $n$ bounding the handlebody $Q_\Sigma$ of the same genus. Now we construct for  $f$ a Morse-Lyapunov function with $4+2n$ critical points. 

Our construction of a quasi-energy function is analogous to the construction of an energy function in \cite{GrLauPo2009}. 
	
\begin{enumerate}
\item Choose an energy function $\varphi _ {p}:U_p\to\mathbb R$ in the neighborhood of each fixed point $p$ of $f$ so that $\varphi _ {p}(p)=\dim\, W^u_p$. Let $B_\omega,\,B_S$ be the 3-balls which are the level sets of respective functions $\varphi_\omega,\,\varphi_S$ and such that $B_S\subset int\,Q_\Sigma$. 
 Choose a tubular neighborhood $T_\sigma$ of the arc $W^u_\sigma\setminus(B_\omega\cup Q_\Sigma)$ so that the handlebody $B_\omega\cup Q_\Sigma\cup T_\sigma$ of genus $n$ is a trapping neighborhood of $\omega$ and its intersection with $W^s_\sigma$ is the 2-disk. Denote by  $P^+$ the smoothing of this body by {addition of a small exterior collar}.

\item Due to \cite[Section 4.3]{GrLauPo2009} there exists an energy function $\varphi:P^+\setminus int\,Q_\Sigma$ whose value on $\partial P^+$ is  $4/3$, whose value on $\bar\Sigma$ is $2/3$ and which has exactly two critical points $\omega,\,\sigma$ of respective Morse indexes $0,1$. The disks  $d_1,\dots,d_{2n-1}$ cut the handlebody $Q_\Sigma$ making the 3-ball. Denote by $B_\Sigma$ the smoothing of this ball by removal of the {interior collar}. The results of the classic Morse theory (see, for example, 
\cite{Milnor2016}) allow to extend the function $\varphi$ to the set $Q_\Sigma\setminus int\,B_\Sigma$ in such way that it has $n$ critical points of Morse index 1, one point lying on each disk $d_1,\dots,d_{2n-1}$, while the value of $\varphi$ on $\partial B_{\Sigma}$ is $1/3$. Due to \cite[Lemma 4.2]{GrLauPo2009} the function $\varphi$ can be extended to the ball $B_\Sigma$ by an energy function with the unique critical point $S$ of Morse index 0. Since $f(Q_\Sigma)\subset int\,B_\Sigma$, the constructed function decreases along the trajectories of the diffeomorphism $f$.

\item Denote by $B^+$ a smooth 3-ball obtained as the union $P^-$ with tubular neighborhoods of the disks $d_2,\dots,d_{2n}$. Then $\partial B^+$ is a 2-sphere which belongs to $W^u_N$ and, hence, bounds a 3-ball $B^-$ there such that $B^+\cup B^-=\mathbb S^3$.  Thus, $P^-=\mathbb S^3\setminus int\,P^+$ is the handlebody of genus $n$.  The results of the classic Morse theory (see, for example, 
\cite{Milnor2016}) allow to extend the function $\varphi$ to the set $P^-\setminus int\,B^-$ in such way that it has $n$ critical points of Morse index 2, one point lying on each disk $d_2,\dots,d_{2n}$, while the value of $\varphi$  on $\partial B^-$ is $5/3$. According to \cite[Lemma 4.2]{GrLauPo2009} the function $\varphi$ can be extended to the ball $B^-$ by an energy function with unique critical point $N$ of Morse index 3. Since $f^{-1}(P^-)\subset int\,B^-$, the constructed function decreases along the trajectories of the diffeomorphism $f$ and, therefore, it is the desired quasi-energy function.
\end{enumerate}

\newpage

\bibliography{biblio}
\bibliographystyle{ijmart}

\end{document}